\documentclass[9pt,journal]{IEEEtran}

\usepackage{cite}
\usepackage{graphicx}
\usepackage{amsmath,amssymb,amsthm}
\usepackage{algorithm}
\usepackage{algpseudocode}
\usepackage{dblfloatfix}
\usepackage{url}

\newtheorem{theorem}{Theorem}

\newtheorem{assumption}{Assumption}
\newtheorem{lemma}{Lemma}

\theoremstyle{definition}
\newtheorem{example}{Example}[section]

\DeclareMathOperator{\diag}{diag}
\newcommand{\vx}{\mathbf{x}}
\newcommand{\vy}{\mathbf{y}}
\newcommand{\vu}{\mathbf{u}}
\newcommand{\vv}{\mathbf{v}}
\newcommand{\vw}{\mathbf{w}}
\newcommand{\ve}{\mathbf{e}}
\newcommand{\vz}{\mathbf{z}}
\newcommand{\mG}{\mathbf{G}}
\newcommand{\mB}{\mathbf{B}}
\newcommand{\mI}{\mathbf{I}}
\newcommand{\mP}{\mathbf{P}}
\newcommand{\mS}{\mathbf{S}}

\newcommand{\mH}{\mathbf{H}}
\newcommand{\mC}{\mathbf{C}}
\newcommand{\mQ}{\mathbf{Q}}
\newcommand{\mR}{\mathbf{R}}
\newcommand{\mW}{\mathbf{W}}
\newcommand{\mV}{\mathbf{V}}
\newcommand{\mTheta}{\mathbf{\Theta}}
\newcommand{\valpha}{\boldsymbol{\alpha}}
\newcommand{\vphi}{\boldsymbol{\varphi}}
\newcommand{\vtheta}{\boldsymbol{\theta}}
\newcommand{\vTheta}{\boldsymbol{\Theta}}
\newcommand{\vPhi}{\boldsymbol{\Phi}}
\newcommand{\vGamma}{\boldsymbol{\Gamma}}
\newcommand{\vbeta}{\boldsymbol{\beta}}
\newcommand{\vmu}{\boldsymbol{\mu}}
\newcommand{\veta}{\boldsymbol{\eta}}
\newcommand{\vU}{\mathbf{U}}
\usepackage{color}

\makeatletter
\long\def\@makecaption#1#2{%
\ifx\@captype\@IEEEtablestring%
\footnotesize\bgroup\par\centering\@IEEEtabletopskipstrut{\normalfont\footnotesize #1}\\{\normalfont\footnotesize\scshape #2}\par\addvspace{0.5\baselineskip}\egroup%
\@IEEEtablecaptionsepspace
\else
\@IEEEfigurecaptionsepspace
\setbox\@tempboxa\hbox{\normalfont\footnotesize {#1.}\nobreakspace\nobreakspace #2}%
\ifdim \wd\@tempboxa >\hsize%
\setbox\@tempboxa\hbox{\normalfont\footnotesize {#1.}\nobreakspace\nobreakspace}%
\parbox[t]{\hsize}{\normalfont\footnotesize\centering\unhbox\@tempboxa#2\par}%
\else
\hbox to\hsize{\normalfont\footnotesize\hfil\box\@tempboxa\hfil}%
\fi
\fi}
\makeatother
\begin{document}

\title{Revealing Strategic Interactions in Network Games Under Decaying Active Probing}

% \author{Xiaoyu~Xin, Longxu~Zhang, Jinlong~Lei, Wenxiao~Zhao, and Yiguang~Hong%
% \thanks{Xiaoyu Xin, Jinlong Lei, and Yiguang Hong are with Shanghai Autonomous Intelligent Unmanned Systems Science Center, Tongji University, Shanghai 201210, China (e-mail: 2151471@tongji.edu.cn; leijinlong@tongji.edu.cn; yghong@tongji.edu.cn).}%
% \thanks{Longxu Zhang and Wenxiao Zhao are with the Academy of Mathematics and Systems Science, Chinese Academy of Sciences, Beijing 100190, China (e-mail: zhanglongxu@amss.ac.cn; wxzhao@amss.ac.cn).}}

\author{
Xiaoyu~Xin\textsuperscript{1},
Longxu~Zhang\textsuperscript{2},
Jinlong~Lei\textsuperscript{1},
Wenxiao~Zhao\textsuperscript{2},
and Yiguang~Hong\textsuperscript{1}
}

\maketitle

\begingroup
\renewcommand{\thefootnote}{}
\footnotetext{
\textsuperscript{1}Xiaoyu Xin, Jinlong Lei, and Yiguang Hong are with the Shanghai Autonomous Intelligent Unmanned Systems Science Center, Tongji University, Shanghai 201210, China (e-mail: 2151471@tongji.edu.cn; leijinlong@tongji.edu.cn; yghong@tongji.edu.cn).

\textsuperscript{2}Longxu Zhang and Wenxiao Zhao are with the Academy of Mathematics and Systems Science, Chinese Academy of Sciences, Beijing 100190, China (e-mail: zhanglongxu@amss.ac.cn; wxzhao@amss.ac.cn).}
\endgroup

\markboth{IEEE Transactions on Automatic Control, submitted for review}%
{Xin \MakeLowercase{\textit{et al.}}: Strategic Interactions Under Decaying Active Probing}

\maketitle
\bstctlcite{BSTcontrol}

\begin{abstract}
Revealing the interaction topology underlying strategic behavior is fundamental to prediction, intervention, and policy design in networked systems. Yet the interaction matrix is often unobservable, and passive observation of repeated actions fails to provide sufficient excitation for reliable recovery. This paper studies topology recovery in repeated linear-quadratic network games under decaying active probing, where probing inputs are injected into a subset of players and the unknown interaction matrix is inferred from the resulting action trajectories. We first characterize a structural recoverability condition that determines when noiseless probing experiments can make the interaction matrix identifiable. We then show that, under suitable stability and controllability assumptions, a concrete decaying probing signal guarantees exact finite-step recovery while preserving convergence of the repeated-play process. To handle decision perturbations, we further develop a reweighted sparse estimator that achieves almost-sure consistency together with finite-time exact support recovery. These results clarify what can be recovered in both noiseless and perturbed settings.
\end{abstract}


\section{Introduction}

\IEEEPARstart{S}{trategic} interactions in networked systems are often governed by latent interdependence structures: one player's decision influences others, which in turn feeds back into future behavior. Revealing such interaction topologies is important for forecasting collective responses, designing interventions, evaluating policies, and pinpointing influential or vulnerable players. These questions arise in social and economic networks, decentralized markets, and engineered multi-player systems, where the effectiveness of prediction and control depends crucially on how strategic couplings are organized \cite{galeotti2020targeting,giannakis2018topology,jackson2015games,parise2023graphon}. For instance, peer effects in social groups, local competition in markets, and coordinated load adjustment in energy systems all induce strategic responses over latent interaction networks \cite{bramoulle2014strategic,schlund2022dynamic,liu2023demandresponse}.

To systematically analyze such hidden strategic couplings, network games provide a natural modeling framework \cite{ballester2006keyplayer,bramoulle2016games,jackson2015games}. In particular, linear-quadratic network games are especially attractive because they retain the essential strategic coupling structure while remaining analytically tractable. Equilibrium actions and repeated responses in these games are tightly linked to the underlying network structure, and many behavioral properties can be characterized through centrality-type quantities \cite{ballester2006keyplayer,bramoulle2014strategic,parise2023graphon}. This close connection between behavior and topology has made topology recovery a central problem in the study of strategic interactions.

In many applications, however, the interaction topology is hidden from direct observation and cannot be read off from prior system knowledge. The strength of peer effects is latent in social systems, substitution and complementarity patterns are hidden in decentralized markets, and coupling channels are only partially observable in engineered strategic systems. As a result, one often has access only to observed actions, rather than to the true interaction matrix itself \cite{pmlr-v119-leng20a,10551735,ghoshal2017behavioral,ghoshal2018polymatrix}. This creates a fundamental inverse problem: can the hidden network of strategic influence be recovered from behavioral data alone?

A growing literature has studied this question from observed player actions \cite{pmlr-v119-leng20a,10551735,ghoshal2017behavioral,ghoshal2018polymatrix,kuleshov2015inverse,11164537}. Existing approaches have mainly focused on passive observations, such as equilibrium snapshots or repeated action trajectories, under parametric or learning-based formulations. Related graph-learning and network-inference works also show how latent structure may be inferred from data, although they do not exploit repeated strategic best responses directly \cite{giannakis2018topology,dong2019learning}. These studies indicate that behavior can reveal latent strategic structure, even when the utility parameters are not fully known. However, passive observation faces a basic limitation in repeated interactions. Players typically do not jump directly to equilibrium; instead, they adjust actions gradually in response to others, generating trajectories that are observable but progressively less informative. As repeated play approaches a steady state, state variations diminish and the data lose excitation, rendering the network structure not recoverable from passive data alone. This motivates the need for active probing mechanisms that deliberately inject informative perturbations into the game to excite the system and reveal the hidden interaction matrix \cite{ljung1999system,chen2012identification}.

Motivated by this observation, we consider active probing of a subset of players and infer the hidden interaction matrix from the induced responses. Under this formulation, successful topology recovery depends not only on the unknown interaction matrix, but also on the probing channel and the excitation pattern supplied to the game. The goal is, therefore, to {investigate} when active probing makes topology recovery possible and  {how} to construct a concrete probing-and-estimation scheme that preserves convergence of the repeated-play process.

The main contributions of this paper are as follows. First of all, we formulate topology recovery under active probing for repeated linear-quadratic network games and characterize a structural recoverability condition for noiseless probing experiments.
     We then construct a decaying probing scheme that preserves convergence of the repeated-play process while still generating sufficiently informative data, and we show that it guarantees exact finite-step recovery of the interaction matrix under suitable conditions. Finally, under decision perturbations, we develop a reweighted sparse estimator and establish almost-sure consistency together with finite-time exact support recovery. 

The remainder of the paper is organized as follows. Section II
formulates the probed repeated-play model. Section III studies struc-
tural recoverability and decaying probing in the noiseless setting. Section IV addresses decision perturbations and sparse recovery. Section V
reports simulation results, and the appendices collect the technical
proofs.

\section{Problem Formulation}

Consider a repeated network game with $N$ players indexed by $V=\{1,2,\dots,N\}$. Let $x_i(t)\in\mathbb{R}$ denote the action of player $i$ at time $t=0,1,\ldots$, and define the collective action vector by
\[
\vx_t=\begin{bmatrix}x_1(t)&x_2(t)&\cdots&x_N(t)\end{bmatrix}^{\top}\in\mathbb{R}^N.
\]
The strategic dependence among players is described by an unknown interaction matrix $\mG=[g_{ij}]\in\mathbb{R}^{N\times N}$, where $g_{ij}$ represents the influence of player $j$ on player $i$. Throughout the paper, we take $g_{ii}=0$ for all $i\in V$. Our goal is to recover $\mG$ from repeated action observations.

In particular, we consider a linear-quadratic network game \cite{ballester2006keyplayer,bramoulle2014strategic}, in which player $i$ chooses a scalar action $x_i\in\mathbb{R}$ to maximize 
$$
J_i(x_i,\vx_{-i})
= \alpha_i x_i - \frac{1}{2}x_i^2
+ \sum_{j\neq i} g_{ij} x_i x_j,
$$
where $\alpha_i\in\mathbb{R}$ is an unknown marginal utility parameter and $\vx_{-i}$ denotes the actions of all other players. Define $\valpha=[\alpha_1,\ldots,\alpha_N]^\top$. 
Hence, at a Nash equilibrium, there holds
$
x_i^\ast = \alpha_i + \sum_{j\neq i} g_{ij}x_j^\ast,
$
and stacking these equations yields
and, in vector form,
\begin{equation}\label{NE}
   (\mI-\mG)\vx^\ast=\valpha.
\end{equation}

We model the repeated play via myopic best-response updates, i.e., at time $t+1$, player $i$ chooses
$
x_i(t+1)
= \arg\max_{x_i} J_i(x_i,\vx_{-i}(t)).
$
By the first-order optimality condition, the resulting dynamics are
$$
x_i(t+1)
= \alpha_i + \sum_{j\neq i} g_{ij}x_j(t),
\quad i\in V,
$$
and, in vector form,
\begin{equation}\label{eq:model_original}
    \vx_{t+1}
    = \valpha + \mG\vx_t.
\end{equation}

When the repeated best-response dynamics are stable, the trajectory approaches the Nash equilibrium $\vx^\ast$ and progressively loses excitation, which makes passive observation insufficient for reliable topology recovery.
To overcome this limitation, we assume that probing inputs can be injected into a subset of players. Let $\vu_t\in\mathbb{R}^N$ denote the probing signal and let $\mB=\diag(b_1,\dots,b_N)$ be a known selection matrix, where $b_i=1$ indicates that player $i$ can be directly excited and $b_i=0$ otherwise. The initial probing input $\vu_0$ is prescribed, and input values before the experiment are taken as zero when finite-history regressors are used. The resulting actively probed repeated-play dynamics are
\begin{equation}\label{eq:model_probed}
    \vx_{t+1}
    =
    \valpha
    +
    \mG\vx_t
    +
    \mB\vu_t.
\end{equation}
The probing signal is the key mechanism that actively excites the repeated game so that the hidden interaction matrix becomes recoverable from the observed responses.

\begin{example}
Consider a residential demand-response program \cite{schlund2022dynamic,liu2023demandresponse}, where each player is a household and $x_i(t)$ denotes its electricity-consumption adjustment relative to a baseline profile. The interaction coefficient $g_{ij}$ is generally unknown and represents how the adjustment of household $j$ affects household $i$. Such influence may reflect peer effects, shared local information, or physical coupling through the distribution network, and it is rarely available to the aggregator in advance. To reveal this hidden structure, the aggregator may apply temporary incentives only to selected households, such as short-term rebates, targeted price discounts, or device-level control commands. These incentives correspond to the probing channel $\mB\vu_t$ in our model. By observing how the induced adjustments propagate through repeated household responses, the aggregator can infer the hidden interaction matrix $\mG$.
\end{example}

% Based on the above formulation, this paper addresses the following two questions.
% \begin{enumerate}
%     \item Under what conditions can the interaction matrix $\mG$ be uniquely recovered from repeated trajectory data generated by the probed game?
%     \item How should the probing input $\vu_t$ be designed so that the interaction matrix $\mG$ can be exactly recovered in finite time while preserving convergence of the repeated-play process?
% \end{enumerate}

\section{Identifiability and Exact Recovery}

This section studies structural recoverability and probing design for the repeated-play model \eqref{eq:model_probed}. We first characterize when a probing experiment can be   informative enough for identification, and then show that a decaying probing signal can guarantee exact finite-step recovery while preserving convergence of the repeated-play process.

\subsection{Recoverability Conditions}

\begin{assumption}\label{ass:noiseless_stability}
The spectral radius of $\mG$ satisfies $\rho(\mG)<1$.
\end{assumption}

Assumption~\ref{ass:noiseless_stability} ensures that the repeated best-response dynamics are stable and admit a unique Nash equilibrium.

\begin{theorem}\label{thm:noiseless_recoverability}
Under Assumption~\ref{ass:noiseless_stability}, there exists an initial state $\vx_0$ and a probing sequence $\vu_t$ for which $\mG$ is uniquely
recoverable from the resulting trajectory  if and only if
\[
\operatorname{rank}\begin{bmatrix}\lambda \mI-\mG, \valpha, \mB\end{bmatrix}\ge N-1,
\quad \forall \lambda\in\mathbb{C},
\]
\end{theorem}
\begin{proof}
We first prove necessity. If there exists $\vx_0$ such that $(\mG,[\vx_0,\valpha,\mB])$ is controllable, then 
\[
\operatorname{rank}
\begin{bmatrix}
\lambda \mI-\mG,  \vx_0,  \valpha,  \mB
\end{bmatrix}
=
N,
\quad \forall \lambda\in\mathbb{C}.
\]
Removing the single column $\vx_0$ can reduce the rank by at most one. Therefore,
\[
\operatorname{rank}
\begin{bmatrix}
\lambda \mI-\mG , \valpha , \mB
\end{bmatrix}
\ge N-1,
\quad \forall \lambda\in\mathbb{C}.
\]

We next prove sufficiency. Suppose, to the contrary, that for every $\vx_0\in\mathbb{R}^N$, the augmented pair $(\mG,[\vx_0,\valpha,\mB])$ is uncontrollable. Then for each $\vx_0\in\mathbb{R}^N$ there exist an eigenvalue $\lambda\in\sigma(\mG)$ and a nonzero vector $q\in\mathbb{C}^N$ such that
\begin{equation*}
q^\top(\lambda \mI-\mG)=0,\quad q^\top\vx_0=0,\quad q^\top\valpha=0,\quad q^\top\mB=0.
\end{equation*}

Let $\sigma(\mG)=\{\lambda_1,\dots,\lambda_r\}$ be the set of distinct eigenvalues of $\mG$. For each $\lambda_i$, define
\[
\mathcal{N}_i
:=
\{q\in\mathbb{C}^N:\ q^\top(\lambda_i \mI-\mG)=0,\ q^\top\valpha=0,\ q^\top\mB=0\}.
\]
Because
\[
\operatorname{rank}
\begin{bmatrix}
\lambda_i \mI-\mG , \valpha , \mB
\end{bmatrix}
\ge N-1,
\]
the left null space $\mathcal{N}_i$ has dimension at most one for every $i$.

For each index $i$ with $\mathcal{N}_i\neq\{0\}$, choose a nonzero vector $q_i$ spanning $\mathcal{N}_i$. The contradiction assumption then implies that for every $\vx_0\in\mathbb{R}^N$, there exists some $i$ such that $q_i^*\vx_0=0$. Hence
\[
\mathbb{R}^N
=
\bigcup_{i:\,\mathcal{N}_i\neq\{0\}}
\{x\in\mathbb{R}^N:\ q_i^* x=0\}.
\]
Each set on the right-hand side is a proper linear subspace of $\mathbb{R}^N$. A finite union of proper linear subspaces cannot cover $\mathbb{R}^N$, which is a contradiction. Therefore, there exists $\vx_0\in\mathbb{R}^N$ such that $(\mG,[\vx_0,\valpha,\mB])$ is controllable.
\end{proof}

\noindent
{Theorem}~\ref{thm:noiseless_recoverability} is a existence result: it tells us when a  probing experiment can be informative enough to recover the interaction matrix from the relation $\vx_{t+1}-\mB\vu_t=\valpha
+\mG\vx_t$. However, it does not by itself give a constructive probing law.
Therefore, we next study a concrete decaying probing scheme under a stronger assumption that ensures the whole network is excited.

\subsection{Decaying Probing Signals and Convergence Preservation}

% The rank condition in Proposition~\ref{thm:noiseless_recoverability} characterizes structural recoverability under a suitable initial state, but it does not by itself provide a constructive probing law that excites the whole network. We therefore next study a concrete decaying probing scheme together with a stronger assumption that makes the subsequent recovery and convergence analysis tractable.

\begin{assumption}\label{ass:noiseless_controllability}
The matrix pair $(\mG,\mB)$ is controllable.
\end{assumption}

Assumption~\ref{ass:noiseless_controllability} is stronger than the condition in {Theorem}~\ref{thm:noiseless_recoverability}. Indeed, if $(\mG,\mB)$ is controllable, then
$
\operatorname{rank}
\begin{bmatrix}
\lambda \mI-\mG , \valpha , \mB
\end{bmatrix}
=N,
$
and {Theorem}~\ref{thm:noiseless_recoverability} is satisfied automatically. The controllability of $(\mG,\mB)$ ensures that every network mode is excited, either directly or through $\mG$, which is crucial for exact finite-step topology recovery.

The proof of exact finite-step recovery relies on two standard auxiliary tools \cite{chen2012identification, kailath1980linear}. We record them next for later use.

\begin{lemma}\label{lem:noiseless_md}
Let $\{X_n,\mathcal{F}_n\}$ be a martingale difference sequence,
and let $\{M_n,\mathcal{F}_n\}$ be an adapted sequence
of random matrices such that $\|M_n\|<\infty$
almost surely for all $n\ge 0$.
Assume that
\begin{equation*}
    \sup_{n\ge 0}
    \mathbb{E}\!\left[
        \|X_{n+1}\|^\alpha \mid \mathcal{F}_n
    \right] < \infty
    \quad \text{a.s.}
\end{equation*}
for some $\alpha\in(0,2]$.
Then for any $\eta>0$,
\begin{equation*}
    \sum_{i=0}^n M_i X_{i+1}
    =
    O\!\left(
        s_n(\alpha)\,
        \log^{\frac{1}{\alpha}+\eta}
        \bigl(s_n^\alpha(\alpha)+e\bigr)
    \right)
    \quad \text{a.s.},
\end{equation*}
where
$
    s_n(\alpha)
    =
    \left(
        \sum_{i=0}^n \|M_i\|^\alpha
    \right)^{\frac{1}{\alpha}}.
$
\end{lemma}

\begin{lemma}\label{lem:noiseless_coprime}
Let $A(s)$ and $B(s)$ be polynomial matrices
with the same number of rows $m$.
Then $A(s)$ and $B(s)$ are left coprime
if and only if
$
    \operatorname{rank}
    \begin{bmatrix}
        A(s) , B(s)
    \end{bmatrix}
    = m,
    \ \forall s\in\mathbb{C}.
$
\end{lemma}

{Since the marginal utility vector $\valpha$ is unknown but time-invariant, we eliminate it through temporal differencing and obtain the incremental model
\begin{equation}\label{non_noise}
    \vx_{t+1}
    = \vx_t
    + \mG\bigl(\vx_t-\vx_{t-1}\bigr)
    + \mB\bigl(\vu_t-\vu_{t-1}\bigr), \quad \forall t\ge 1
\end{equation}}
Define
$
\vy_{t+1}:=\vx_{t+1}-\vx_t,
$
and set
$
\vtheta^\top
=
\begin{bmatrix}
\mG , \mB , -\mB
\end{bmatrix},
\
\vphi_t
=
\begin{bmatrix}
\vy_t^\top , \vu_t^\top , \vu_{t-1}^\top
\end{bmatrix}^\top.
$
Then the noiseless incremental dynamics {\eqref{non_noise}} can be written as the linear regression
$
\vy_{t+1}=\vtheta^\top\vphi_t.
$
Therefore, if the information matrix
$
\sum_{t=1}^n \vphi_t\vphi_t^\top
$
is nonsingular, ordinary least squares exactly recovers
$
\vtheta^\top
=
\left(\sum_{t=1}^n\vy_{t+1}\vphi_t^\top\right)
\left(\sum_{t=1}^n\vphi_t\vphi_t^\top\right)^{-1},
$
and hence exactly recovers the interaction matrix $\mG$ from the first block of $\vtheta^\top$.
The next two theorems show that the decaying probing law achieves two objectives simultaneously. On the one hand, the probing amplitude vanishes over time and, therefore, does not alter the long-run behavior of repeated play. On the other hand, it still injects enough excitation into the incremental regression to make the hidden interaction matrix identifiable from finitely many samples almost surely. We begin with the exact finite-step recovery statement in the noiseless setting.
\begin{theorem}\label{thm:noiseless_exact_recovery}
Under Assumptions~\ref{ass:noiseless_stability} and \ref{ass:noiseless_controllability}, let $\vu_t=\frac{\vv_t}{t^{\frac{\epsilon}{2}}}$, where $\vv_t \sim \mathcal{N}(0, \mI)$ and $\epsilon \in [0,\frac{1}{N+1}]$. Then there exists a subset $\Omega_0$ of the probability space $\Omega$ with probability measure $\mathbb{P}(\Omega_0) = 1$ such that, for every $\omega \in \Omega_0$, there exists a finite integer threshold $n(\omega)$ for which the matrix $\mG$ can be exactly recovered from the trajectory generated by {\eqref{eq:model_probed}} for all sample sizes $n \ge n(\omega)$.
\end{theorem}

% Theorem~\ref{thm:noiseless_exact_recovery} shows that the probing sequence can be chosen to satisfy two goals simultaneously: it decays fast enough not to destabilize the repeated-play process, yet it still injects enough information to make the information matrix invertible after finitely many samples almost surely.

\begin{proof}
By the preceding regression representation, exact finite-step recovery is reduced to showing that the information matrix
\(
\sum_{t=1}^n \vphi_t \vphi_t^\top
\)
becomes nonsingular after finitely many samples almost surely. The key point is that, although the probing input decays, it does not decay so fast as to destroy excitation in the regression data. Appendix~\ref{app:proof_exact_recovery} establishes this by combining a lower bound on the information matrix with a polynomial-matrix contradiction argument: any asymptotically degenerate direction would induce a polynomial relation incompatible with the controllability of $(\mG,\mB)$. Therefore, the information matrix is eventually invertible almost surely, and ordinary least squares then exactly recovers $\mG$ after a finite time.
\end{proof}

Theorem~\ref{thm:noiseless_exact_recovery} shows that the proposed probing law remains informative enough for exact topology recovery. The next theorem shows that the same probing law is also asymptotically benign, namely that the repeated-play process still converges to the Nash equilibrium.

\begin{theorem}\label{thm:noiseless_convergence}
Under Assumptions~\ref{ass:noiseless_stability}-\ref{ass:noiseless_controllability}, with $\vu_t=\frac{\vv_t}{t^{\frac{\epsilon}{2}}}$, where $\vv_t \sim \mathcal{N}(0, \mI)$ and $\epsilon \in [0,\frac{1}{N+1}]$, it holds that
$
\vx_t\to \vx^\ast \quad \text{a.s.}
$
Moreover, there exists an a.s.\ finite random variable $C_1(\omega)>0$ such that, for all sufficiently large $t$,
\begin{equation}\label{conv0}
    \|\vx_t-\vx^\ast\|
    \le C_1(\omega)\!\left(\rho^t+\frac{\sqrt{\log t}}{t^{\epsilon/2}}\right),
\end{equation}
where $\rho\in(0,1)$.
\end{theorem}
\begin{proof}
    Define $\ve_t=\vx_t-\vx^\ast$. For the probed dynamics \eqref{eq:model_probed}, {by using \eqref{NE}, we derive $  \ve_{t+1}=\mG\ve_t+\mB\vu_t.$}
    Because $\mG$ is stable {by Assumption \ref{ass:noiseless_stability}}, there exists a constant $\rho\in(0,1)$
    such that $\|\mG^k\|<C_G\rho^k$. Thus,
    \begin{equation*}
        \|\ve_t\|
        \leq
        C\rho^t\|\ve_0\|
        +
        C\sum_{k=1}^{t}\rho^{t-k}\frac{\|\vv_k\|}{k^{\frac{\epsilon}{2}}}.
    \end{equation*}
    Since $\vv(k) \sim \mathcal{N}(0, \mI)$, it follows from standard Gaussian tail bounds and the Borel--Cantelli lemma that there almost surely exists a positive random constant $C(\omega)$ such that $\lVert \vv(k) \rVert \le C(\omega) \sqrt{\log k}$ for sufficiently large $k$ \cite{durrett2019probability}.
    Let $a_k=\frac{\sqrt{\log{k}}}{k^{\frac{\epsilon}{2}}}$. Then we obtain
    \begin{equation}\label{conv1}
         \|\ve_t\|
         \leq
         C\rho^t\|\ve_0\|
         +
         C\sum_{k=1}^{t}\rho^{t-k}a_k.
    \end{equation}
    Since $a_k$ is eventually decreasing and converges to zero, the exponentially weighted convolution above satisfies the standard estimate
$
\sum_{k=1}^{t} \rho^{t-k} a_k = O(a_t).
$
This together with \eqref{conv1} proves \eqref{conv0},
and hence $\vx_t\to \vx^\ast$ almost surely.
\end{proof}

\section{Topology Recovery Under Decision Perturbations}

We now turn to the case of decision perturbations. In this setting, exact finite-step recovery of the full interaction matrix is no longer attainable in general. The objective instead becomes asymptotically consistent estimation of $\mG$ together with finite-time exact support recovery. To avoid the noise correlation created by temporal differencing, we work directly with the noisy first-order dynamics and jointly estimate the intercept and interaction matrix. The weighted $\ell_1$ regularization used below is inspired by Lasso-type sparse estimation, adaptive reweighting ideas, and sparse data-driven estimation methods for stochastic dynamical models \cite{tibshirani1996lasso,zou2006adaptive,buhlmann2011statistics,zhao2020sparse}. The perturbed repeated-play model becomes
\begin{equation}\label{exist_noise}
    \vx_{t+1}
    = \valpha
    + \mG \vx_t
    + \mB \vu_t
    + \vw_{t+1}.
\end{equation}

We introduce the following standard assumption on the noise sequence \cite{zhang2023multitask,gan2023distributed,zhu2024distributedls}.
\begin{assumption}\label{ass:perturbed_noise}
The noise sequence $\{\vw_t\}_{t\ge 1}$ consists of i.i.d.\ bounded random
vectors with zero mean. That is,
\[
\mathbb{E}[\vw_t]=0,
\qquad t\ge 1,
\]
and there exists a deterministic constant $C_w>0$ such that
\begin{equation}\label{eq:noise_bounded}
    \|\vw_t\|\le C_w,
    \quad \forall t\ge 1,
    \quad \text{a.s.}
\end{equation}
\end{assumption}

% Under Assumption~\ref{ass:perturbed_noise}, the appropriate target is no longer finite-step exact recovery of every entry of $\mG$, but rather almost-sure consistency together with finite-time exact support recovery. To this end, we adopt a sparse recovery procedure, summarized in Algorithm~1. Let $\widehat{\mTheta}_{n+1}:=[\hat{\valpha}_{n+1},\hat{\mG}_{n+1}]$ denote the output of Algorithm~1. The next theorem states that its interaction block $\hat{\mG}_{n+1}$ achieves both almost-sure consistency and finite-time exact support recovery.
Under Assumption~\ref{ass:perturbed_noise}, the appropriate target is no longer finite-step exact recovery of every entry of $\mG$, but rather almost-sure consistency together with finite-time exact support recovery. To this end, we adopt a sparse recovery procedure summarized in Algorithm~1. Under the conditions of Theorem \ref{thm:perturbed_sparse_recovery}, the information matrix becomes nonsingular after a finite random time almost surely, so $\lambda_{\min}(n)>0$ for all sufficiently large $n$. Thus, Step 8 of Algorithm~1 is a well-defined convex optimization problem whose solution gives the estimator. Let $\widehat{\mTheta}_{n+1}:=[\hat{\valpha}_{n+1},\hat{\mG}_{n+1}]$ denote the output of Alg.~1. 
The next theorem states that its interaction block $\hat{\mG}_{n+1}$ achieves both almost-sure consistency and finite-time exact support recovery.

\begin{algorithm}
    \caption{Finite-Time Exact Support Recovery Algorithm}
    \begin{algorithmic}[1]
    \State Initialize $\mP_0=\alpha_0 \mI_{N+1}$ and $\vTheta_0=0$.
\While{$k = 1, \dots, n$}
    \State Form the adjusted observation and regressor
    \begin{equation*}
        \hat{\vx}_{k+1} = \vx_{k+1}-\mB\vu_k,
    \end{equation*}
    \begin{equation*}
        \vz_k
        =
        \begin{bmatrix}
            1\\
            \vx_k
        \end{bmatrix}.
    \end{equation*}
    \State Update the recursive least-squares iterate
    \begin{equation*}
        \mP_{k+1}
        =
        \mP_k
        -
        \frac{\mP_k \vz_k \vz_k^\top \mP_k}{1 + \vz_k^\top \mP_k \vz_k},
    \end{equation*}
    \begin{equation*}
        \vTheta_{k+1}
        =
        \vTheta_k
        +
        \mP_{k+1} \vz_k
        \bigl(
            \hat{\vx}_{k+1}^\top - \vz_k^\top \vTheta_k
        \bigr).
    \end{equation*}
\EndWhile
    \State Set
    \begin{equation*}
        \lambda_{\max}(n)
        \triangleq
        \lambda_{\max} \left\{ \sum_{k=1}^{n} \vz_k \vz_k^\top \right\}
    \end{equation*}
    \begin{equation*}
        \lambda_{\min}(n)
        \triangleq
        \lambda_{\min} \left\{ \sum_{k=1}^{n} \vz_k \vz_k^\top \right\}.
    \end{equation*}
    Define
    \begin{equation*}
        \delta_n
        \triangleq
        \sqrt{\frac{\log \lambda_{\max}(n)}{\lambda_{\min}(n)}}.
    \end{equation*}
    Let $\widetilde{\mTheta}_{n+1}\triangleq \vTheta_{n+1}^\top$ and define the shifted pilot estimator
    \begin{equation*}
        \widehat{\vTheta}_{n+1}(s, t)
        \triangleq
        \widetilde{\mTheta}_{n+1}(s, t)
        +
        \operatorname{sgn}(\widetilde{\mTheta}_{n+1}(s, t))\delta_n.
    \end{equation*}
    \State Choose a positive sequence $\{\lambda_n\}_{n \ge 1}$ satisfying
    \begin{equation*}
        \lambda_n^2
        =
        \lambda_{\max}(n)
        \sqrt{\log(\lambda_{\max}(n))\lambda_{\min}(n)}.
    \end{equation*}
    \State Solve the convex program
    \begin{equation*}
        J_{n+1}(\vGamma)
        \triangleq
        \sum_{k=1}^{n}
        \|
            \hat{\vx}_{k+1} - \vGamma \vz_k
        \|^2
        +
        \lambda_n
        \sum_{s=1}^N \sum_{t=1}^{N+1}
        \frac{|\vGamma(s, t)|}{|\widehat{\vTheta}_{n+1}(s, t)|}
    \end{equation*}
    and set
    \begin{equation*}
        \vGamma_{n+1}
        \triangleq
        \arg\min J_{n+1}(\vGamma).
    \end{equation*}
    \State \textbf{return} $\widehat{\mTheta}_{n+1}:=\vGamma_{n+1}=[\hat{\valpha}_{n+1},\hat{\mG}_{n+1}]$
\end{algorithmic}
\end{algorithm}

\begin{theorem}\label{thm:perturbed_sparse_recovery}
    Under Assumptions~\ref{ass:noiseless_stability}, \ref{ass:noiseless_controllability} and \ref{ass:perturbed_noise}, let $\vu_t=\frac{\vv_t}{t^{\frac{\epsilon}{2}}}$, where $\vv_t \sim \mathcal{N}(0, \mI)$ and $\epsilon \in [0, \frac{1}{3(N+2)})$. There exists a subset $\Omega_0$ of the probability space $\Omega$ with probability measure $\mathbb{P}(\Omega_0) = 1$, such that for every $\omega \in \Omega_0$, $\|\hat{\mG}_{n+1}-\mG\|\to 0$ as $n\to \infty$. Moreover, there exists a finite integer threshold $n_0(\omega)$ such that, for all sample sizes $n \geq n_0(\omega)$, Algorithm~1 recovers the exact support of the true network matrix $\mG$.
\end{theorem}

\begin{proof}
The proof proceeds in three steps. First, working directly with the first-order perturbed model \eqref{exist_noise}, we rewrite the observations as a joint linear regression in $[\valpha,\mG]$ and prove, by a block-excitation argument, that the associated information matrix remains uniformly nondegenerate almost surely under decaying probing.
Secondly, this lower bound implies almost-sure consistency of the pilot recursive least-squares estimator and, in turn, of the weighted sparse estimator produced by Algorithm~1.
Finally, on the true zero coordinates, the adaptive penalty eventually dominates the residual estimation error, so false positives cannot persist beyond a finite random time. The complete proof is given in Appendix~\ref{app:proof_perturbed_sparse_recovery}.
\end{proof}
Theorem~\ref{thm:perturbed_sparse_recovery} shows that decision perturbations change the attainable notion of recovery: finite-step exact identification of the full matrix is generally lost, but the topology remains recoverable through finite-time exact support recovery, while the matrix estimate itself converges almost surely to the truth. The next theorem shows that, even under decision perturbations, the repeated-play process still converges to the Nash equilibrium in Ces\`aro average.

\begin{theorem}\label{thm:perturbed_convergence}
    Under Assumptions~\ref{ass:noiseless_stability} and \ref{ass:perturbed_noise}, let $\vu_t=\frac{\vv_t}{t^{\epsilon/2}}$, where $\vv_t\sim\mathcal{N}(0,\mI)$ and $\epsilon>0$. For model \eqref{exist_noise}, there holds
\begin{equation}\label{them_con2}
\frac{1}{t}\sum_{s=1}^t \vx_s \to \vx^\ast
\quad \text{a.s.}
\end{equation}
Moreover, for any $\delta>0$, there exists an a.s.\ finite random variable $C_2(\omega,\delta)>0$ such that, for all sufficiently large $t$,
\begin{equation*}
\left\|
\frac{1}{t}\sum_{s=1}^t \vx_s-\vx^\ast
\right\|
\le
C_2(\omega,\delta)
\left(
t^{-\frac{1}{2}+\delta}
+\frac1t\sum_{s=1}^t \frac{\sqrt{\log s}}{s^{\epsilon/2}}
+\frac1t
\right).
\end{equation*}
\end{theorem}
\begin{proof}
Let $\ve_t=\vx_t-\vx^\ast$. Then
\begin{equation*}
\begin{aligned}
\ve_{t+1}
&=\mG\ve_t+\mB\vu_t+\vw_{t+1} \\
&=\mG^{t+1}\ve_0
+\sum_{k=0}^{t}\mG^{\,t-k}\mB\vu_k
+\sum_{k=0}^{t}\mG^{\,t-k}\vw_{k+1}.
\end{aligned}
\end{equation*}
Define $\bar{\ve}_t:=\frac1t\sum_{k=0}^{t-1}\ve_k$, {$\mH_m:=\sum_{j=0}^m \mG^j,$ and  $\mH_\infty:=(\mI-\mG)^{-1}.$ Then}
\[
\bar{\ve}_t
=
\frac1t\sum_{k=0}^{t-1}\mG^k\ve_0
+
\frac1t\sum_{s=1}^{t-1}\mH_{t-1-s}\mB\vu_s
+
\frac1t\sum_{s=1}^{t-1}\mH_{t-1-s}\vw_s.
\]
Since $\rho(\mG)<1$, there exist $C>0$ and $r\in(0,1)$ such that $\|\mG^k\|\le Cr^k$, $\sup_m\|\mH_m\|<\infty$, and $\|\mH_\infty-\mH_m\|\le Cr^m$. Hence the initial term is $O(t^{-1})$.

For the probing term, the same Gaussian bound and convolution estimate used in the proof of Theorem~\ref{thm:noiseless_convergence} imply
\[
\left\|
\frac1t\sum_{s=1}^{t-1}\mH_{t-1-s}\mB\vu_s
\right\|
=
O\!\left(
    \frac1t\sum_{s=1}^{t}\frac{\sqrt{\log s}}{s^{\epsilon/2}}
\right)
\quad \text{a.s.};
\]
see also \cite{zhao2020sparse,gan2023distributed} for analogous geometric-convolution estimates.

For the noise term, write
\[
\frac1t\sum_{s=1}^{t-1}\mH_{t-1-s}\vw_s
=
\mH_\infty \frac1t\sum_{s=1}^{t-1}\vw_s
-\frac1t\sum_{s=1}^{t-1}(\mH_\infty-\mH_{t-1-s})\vw_s.
\]
Because \(\{\vw_t\}\) is i.i.d., bounded, and zero mean, standard martingale rate results \cite{hallheyde1980martingale} imply
\[
\left\|
\frac1t\sum_{s=1}^{t}\vw_s
\right\|
=
O\!\left(t^{-\frac12+\delta}\right)
\quad \text{a.s.}
\]
for every $\delta>0$. Moreover, by \eqref{eq:noise_bounded} and geometric summability,
\[
\begin{aligned}
\left\|
\frac1t\sum_{s=1}^{t-1}(\mH_\infty-\mH_{t-1-s})\vw_s
\right\|
&\le
\frac{C}{t}
\sum_{s=1}^{t-1} r^{t-1-s}\|\vw_s\| \\
&\le
\frac{CC_w}{t}\sum_{j=0}^{\infty} r^j
=
O(t^{-1})
\quad \text{a.s.}
\end{aligned}
\]

Combining the above estimates yields
\[
\left\|
\frac1t\sum_{s=1}^t \vx_s-\vx^\ast
\right\|
\le
C_2(\omega,\delta)
\left(
    t^{-\frac12+\delta}
    +\frac1t\sum_{s=1}^{t}\frac{\sqrt{\log s}}{s^{\epsilon/2}}
    +\frac1t
\right)
\]
for all sufficiently large $t$. Hence, \eqref{them_con2} is established.
\end{proof}
\section{Simulation}

We illustrate the proposed recovery method on a six-player network with interaction matrix
\[
\mG=\begin{bmatrix}
0 & 0.18 & 0 & 0 & 0 & 0\\
0.12 & 0 & -0.15 & 0 & 0 & 0\\
0 & 0.10 & 0 & 0.14 & 0 & 0\\
0 & 0 & 0.16 & 0 & -0.10 & 0\\
0 & 0 & 0 & 0.13 & 0 & 0.11\\
0.09 & 0 & 0 & 0 & 0.12 & 0
\end{bmatrix}.
\]
In all simulations, we set $\mB=\mI$, so every player can be directly probed. In the noiseless case, the probing signal is chosen as $\vu_t=\frac{\vv_t}{(t+1)^{\epsilon/2}}$ with $\epsilon=0.12$. In the noisy {case, we set  $\epsilon=0.03$ and additionally add Gaussian disturbances with standard deviation $0.03$.}

Fig.~\ref{fig:noiseless_simulation} reports the noiseless case. The relative estimation error decreases steadily as more data are collected, and the support accuracy eventually reaches one, which is consistent with exact recovery in the noiseless model.

\begin{figure}[htbp]
    \centering
    \includegraphics[width=0.93\columnwidth]{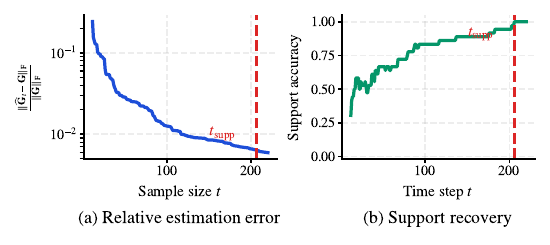}
    \caption{Noiseless simulation results.}
    \label{fig:noiseless_simulation}
\end{figure}

Fig.~\ref{fig:noisy_support} compares the support patterns in the noisy case. Here the estimates are computed directly from the first-order perturbed model \eqref{exist_noise}: the least-squares baseline jointly estimates $[\valpha,\mG]$ from $\vx_{t+1}-\mB\vu_t=\valpha+\mG\vx_t+\vw_{t+1}$, whereas the blue panel applies a finite-sample implementation of the weighted sparse recovery procedure in Algorithm~1 to the same regression. A direct least-squares estimate produces many small but nonzero spurious entries. If an edge is declared whenever $|g_{ij}|>10^{-3}$, the least-squares estimate contains 27 detected edges, while the reweighted sparse estimate contains 11 edges, matching the true network support.

\begin{figure}[htbp]
    \centering
    \includegraphics[width=0.98\columnwidth]{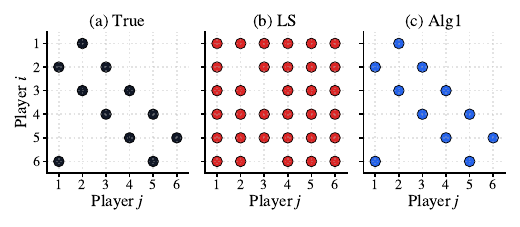}
    \caption{Support comparison in the noisy case}
    \label{fig:noisy_support}
\end{figure}

{Fig.~\ref{fig:nash_convergence}  further the convergence to the Nash equilibrium. In the noiseless case, the action profile $\vx_t$ approaches $\vx^{\ast}$. In the noisy case, the instantaneous state fluctuates because of the disturbances, but the time average
$
\bar{\vx}_t=\frac{1}{t}\sum_{s=1}^{t}\vx_s
$
converges to $\vx^{\ast}$, which agrees with the theoretical result.}

\begin{figure}[htbp]
    \centering
    \includegraphics[width=0.93\columnwidth]{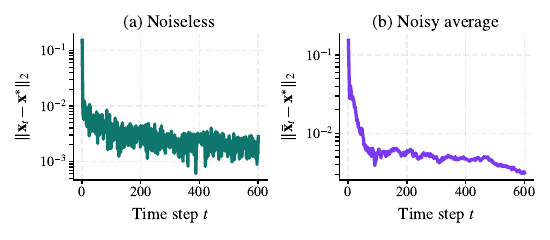}
    \caption{Convergence to the Nash equilibrium.}
    \label{fig:nash_convergence}
\end{figure}

\section{Conclusion} 

This paper studied topology recovery in repeated linear-quadratic network games under decaying active probing. We first characterized a structural recoverability condition showing when noiseless probing experiments can make the hidden interaction matrix identifiable. We then proposed a concrete decaying probing signal that preserves convergence of repeated play while guaranteeing exact finite-step recovery in the noiseless setting under suitable stability and controllability assumptions. Under decision perturbations, we further developed a reweighted sparse estimator and established almost-sure consistency together with finite-time exact support recovery.
Overall, the results show that carefully designed active probing can overcome the loss of excitation inherent in passive repeated-play observations and can reveal latent strategic interactions while preserving convergence of the repeated-play process. Future work may consider partial observation of player actions, time-varying or nonlinear strategic couplings, and probing design under stricter actuation constraints.

\appendices
\section{}\label{app:proof_exact_recovery}
\begin{proof}[Proof of Theorem~\ref{thm:noiseless_exact_recovery}]
\mbox{}

To establish finite-step recoverability, it suffices to show that, for any $\alpha\in(\frac{1}{2},1-\epsilon)$,
\begin{equation*}
    \liminf_{n\to\infty}
    \frac{1}{n^\alpha}
    \lambda_{\min}
    \!\left(
        \sum_{t=1}^n
        \vphi_t \vphi_t^\top
    \right)>0.
\end{equation*}
For subsequent analysis, define
$
\det(\mI-\mG z)=a_0+a_1z+\ldots+a_N z^N
$
and
\begin{equation*}
    f_t
    =
    \det(\mI-\mG z)\,\vphi_t
    =
    \sum_{j=0}^N a_j\vphi_{t-j}.
\end{equation*}
For convenience, define
$\lambda_{\min}(n)
=
\lambda_{\min}\!\left(
    \sum_{t=1}^n
    \vphi_t \vphi_t^\top
\right)$.
By the definition of the minimum eigenvalue,
\begin{equation*}
    \lambda_{\min}\!\left(\sum_{t=1}^n f_tf_t^\top\right)
    =
    \inf_{\|x\|=1}\sum_{t=1}^n (x^\top f_t)^2
    \leq
    (N+1) \sum_{j=0}^N a_j^2\lambda_{\min}(n).
\end{equation*}
Therefore, it suffices to prove that
\begin{equation}\label{inverse}
    \liminf_{n\to\infty}
    \frac{1}{n^\alpha}
    \lambda_{\min}\!\left(
        \sum_{t=1}^n f_t f_t^\top
    \right)>0.
\end{equation}

We argue by contradiction. Suppose that the claim {\eqref{inverse}} does not hold. Then there exists a subsequence $\{n_k\}$ and a corresponding sequence of unit vectors $\{\veta_{n_k}\}$, where
$
    \veta_{n_k}
    =
    \bigl(
        p_{n_k}^\top,\,
        (q_{n_k}^{(0)})^\top,\,
        (q_{n_k}^{(1)})^\top
    \bigr)^\top,
$
such that
\begin{equation}\label{eq:appendix_contradiction}
    \lim_{k\to\infty}
    \frac{1}{n_k^\alpha}
    \sum_{i=1}^{n_k}
    (\veta_{n_k}^\top f_i)^2
    =0.
\end{equation}

Let $A(z)=\mI-\mG z$ and $\mathcal{B}(z)=\mB(\mI-z\mI)$. For each $n_k$, define the polynomial row vector
\begin{equation*}
\begin{aligned}
\mH_{n_k}(z)
&=
p_{n_k}^\top (\operatorname{Adj} A(z)) z\mathcal{B}(z)
+ \sum_{i=0}^{1} (q_{n_k}^{(i)})^\top z^i [(\det A(z)) \mI] \\
&= \sum_{j=0}^{N+1} h_{n_k}^{(j)} z^j.
\end{aligned}
\end{equation*}

Since
$
(\det A(z))\,\vy_n = (\operatorname{Adj} A(z)) z\mathcal{B}(z)\,\vu_n,
$
we obtain
\begin{equation*}
\begin{aligned}
\veta_{n_k}^\top f_i
&=
\bigl\{
p_{n_k}^\top(\operatorname{Adj} A(z)) z\mathcal{B}(z)
+ (q_{n_k}^{(0)})^\top (\det A(z))\mI \\
&
+ (q_{n_k}^{(1)})^\top z (\det A(z))\mI
\bigr\}\vu_i \\
&=
\mH_{n_k}(z)\vu_i
=
\sum_{j=0}^{N+1} h_{n_k}^{(j)} \vu_{i-j}.
\end{aligned}
\end{equation*}
Hence, by \eqref{eq:appendix_contradiction},
$
\lim_{k\to \infty}
n_k^{-\alpha}
\sum_{i=1}^{n_k}
\left(
\sum_{j=0}^{N+1} h_{n_k}^{(j)} \vu_{i-j}
\right)^2
= 0.
$
By Lemma~\ref{lem:noiseless_md}, for any $\delta>0$,
$
\left\|
\sum_{i=1}^n \vu_{i-j} \vu_i^\top
\right\|
=
O\bigl(n^{\frac{1+\delta}{2}}\bigr).
$
Because $\alpha>\frac{1}{2}$, the cross terms are negligible relative to $n^\alpha$, and we obtain
\begin{equation*}
    \lim_{k\to \infty}
    \frac{1}{n_k^\alpha}
    \sum_{i=1}^{n_k}
    ({h_{n_k}^{(0)}} \vu_i)^2
    =0.
\end{equation*}
Since $\vu_i=\frac{\vv_i}{i^{\frac{\epsilon}{2}}}$ with $\vv_i\sim\mathcal{N}(0,\mI)$, this yields
\begin{equation*}
    \|h_{n_k}^{(0)}\|^2
    =
    o\bigl(n_k^{-(1-\epsilon-\alpha)}\bigr).
\end{equation*}
Repeating the same argument for the remaining coefficients gives
\begin{equation*}
    \lim_{k\to\infty}\|h_{n_k}^{(j)}\| = 0,
    \quad \forall j\in \{0,1,\ldots,N+1\}.
\end{equation*}
Hence $\mH_{n_k}(z)\to 0$. Since the sequence $\veta_{n_k}$ is bounded, it admits a convergent subsequence. Let
\[
\veta_0=(p^\top,(q^{(0)})^\top,(q^{(1)})^\top)^\top
\]
denote its limit with $\|\veta_0\|=1$. Passing to the limit gives
\begin{equation}\label{eq:appendix_poly_identity}
    p^\top(\operatorname{Adj}(A(z)))z\mathcal{B}(z)
    =
    -\sum_{i=0}^1 (q^{(i)})^\top z^i \det(A(z))\mI.
\end{equation}

Next, we show that the polynomial matrices $A(z)$ and $\mathcal{B}(z)$ are left coprime. By Lemma~\ref{lem:noiseless_coprime}, this is equivalent to
\begin{equation*}
    \operatorname{rank}
    \left(
    \begin{bmatrix}
    \mI-\mG z , \mB(\mI-z\mI)
    \end{bmatrix}
    \right)
    = N,
    \quad \forall z\in\mathbb{C}.
\end{equation*}
We verify this in three cases.

If $z=1$, then
$
\begin{bmatrix}
\mI-\mG , 0
\end{bmatrix}
$
has full row rank because $\mI-\mG$ is nonsingular under Assumption~\ref{ass:noiseless_stability}.
If $z=0$, then
$
\begin{bmatrix}
\mI , \mB
\end{bmatrix}
$
has a full row rank trivially.
If $z\neq 0,1$, let $\lambda=z^{-1}$. Then
\begin{equation*}
    \operatorname{rank}
    \left(
    \begin{bmatrix}
    \mI-\lambda^{-1}\mG , \mB(\mI-z\mI)
    \end{bmatrix}
    \right)
    =
    \operatorname{rank}
    \left(
    \begin{bmatrix}
    \lambda \mI-\mG , \mB
    \end{bmatrix}
    \right).
\end{equation*}
By the PBH test, Assumption~\ref{ass:noiseless_controllability} implies
\[
\operatorname{rank}
\left(
\begin{bmatrix}
\lambda \mI-\mG , \mB
\end{bmatrix}
\right)=N,
\quad \forall \lambda\in\mathbb{C}.
\]
Hence $A(z)$ and $\mathcal{B}(z)$ are left coprime.

Therefore, there exist polynomial matrices $M(z)$ and $N(z)$ such that
\begin{equation}\label{eq:appendix_bezout}
    A(z)M(z) + \mathcal{B}(z)N(z) = \mI.
\end{equation}
Multiplying \eqref{eq:appendix_bezout} by $p^\top\operatorname{Adj}(A(z))$ and using \eqref{eq:appendix_poly_identity}, we obtain
\begin{equation*}
\begin{aligned}
    p^\top\operatorname{Adj}(A(z))
    &=
    p^\top\operatorname{Adj}(A(z))(A(z)M(z)+\mathcal{B}(z)N(z)) \\
    &=
    \det(A(z))
    \left(
        p^\top M(z)
        -
        \sum_{i=0}^1
        (q^{(i)})^\top z^i N(z)
    \right).
\end{aligned}
\end{equation*}
Hence there exists a polynomial row vector $\sum_{j=0}^{\ell} (\vmu^j)^\top z^j$ such that
\begin{equation}\label{eq:appendix_rank1}
    p^\top
    =
    \left(
        \sum_{j=0}^{\ell} (\vmu^j)^\top z^j
    \right)A(z),
\end{equation}
and
\begin{equation}\label{eq:appendix_rank2}
    \sum_{i=0}^1 (q^{(i)})^\top z^i
    =
    \left(
        \sum_{j=0}^{\ell} (\vmu^j)^\top z^j
    \right) z\mathcal{B}(z).
\end{equation}
Expanding \eqref{eq:appendix_rank1} and \eqref{eq:appendix_rank2} gives
\begin{equation*}
    p^\top
    =
    \sum_{j=0}^{\ell} (\vmu^j)^\top z^j
    -
    \sum_{j=0}^{\ell} (\vmu^j)^\top \mG z^{j+1},
\end{equation*}
and
\begin{equation*}
    \sum_{i=0}^1 (q^{(i)})^\top z^i
    =
    \sum_{j=0}^{\ell} (\vmu^j)^\top \mB z^{j+1}
    -
    \sum_{j=0}^{\ell} (\vmu^j)^\top \mB z^{j+2}.
\end{equation*}
Because the left-hand sides have degrees at most $0$ and $1$, respectively, the coefficients of the highest powers of $z$ imply
\[
(\vmu^\ell)^\top \mG=0,
\qquad
(\vmu^\ell)^\top \mB=0.
\]
By the PBH test at $\lambda=0$ and Assumption~\ref{ass:noiseless_controllability}, this yields $\vmu^\ell=0$. Repeating the same argument recursively for the next highest powers shows that $\vmu^j=0$ for every $j=0,\dots,\ell$. Consequently,
\[
p=0,
\qquad
q^{(0)}=0,
\qquad
q^{(1)}=0,
\]
which contradicts $\|\veta_0\|=1$.

Therefore, the information matrix $\sum_{t=1}^n \vphi_t \vphi_t^\top$ is invertible after finitely many steps almost surely, and the matrix $\mG$ can then be exactly recovered by ordinary least squares.
% \begin{equation*}
%     \vtheta
%     =
%     \left(\sum_{t=1}^n\vphi_t\vphi_t^\top\right)^{-1}
%     \left(\sum_{t=1}^n\vphi_t \vy_{t+1}^\top\right).
% \end{equation*}
\end{proof}

\section{}\label{app:proof_perturbed_sparse_recovery}
\begin{proof}[Proof of Theorem~\ref{thm:perturbed_sparse_recovery}]
\mbox{}

We divide the proof into three steps. Step 1 establishes the lower bound for the information matrix. Step 2 proves almost-sure consistency. Step 3 proves finite-time exact support recovery.

For the noisy first-order model \eqref{exist_noise}, define
\[
\hat{\vx}_{t+1}:=\vx_{t+1}-\mB\vu_t,
\qquad
\vz_t:=
\begin{bmatrix}
    1\\
    \vx_t
\end{bmatrix},
\qquad
\mTheta^\ast:=
\begin{bmatrix}
    \valpha & \mG
\end{bmatrix}.
\]
Then
$
\hat{\vx}_{t+1}=\mTheta^\ast \vz_t+\vw_{t+1}.
$
Let
\begin{equation*}
\mS_n^{(z)}
\triangleq
\sum_{t=1}^{n}\vz_t\vz_t^\top,
\lambda_{\min}(n)
\triangleq
\lambda_{\min}\!\bigl(\mS_n^{(z)}\bigr),
\lambda_{\max}(n)
\triangleq
\lambda_{\max}\!\bigl(\mS_n^{(z)}\bigr).
\end{equation*}

It suffices to prove that, for any
\(
\alpha\in(\frac{1}{2},1-\epsilon)
\),
\begin{equation}\label{eq:appendix_first_order_lower_bound}
    \liminf_{n\to\infty}
    \frac{1}{n^\alpha}
    \lambda_{\min}\!\bigl(\mS_n^{(z)}\bigr)
    >0
    ,\ \text{a.s.}
\end{equation}

Because \(\rho(\mG)<1\), iterating \eqref{exist_noise} and using
\(
\sum_{s=1}^{n}\|\vu_s\|^2=O(n^{1-\epsilon})
\)
almost surely together with the boundedness of \(\{\vw_t\}\) gives
\begin{equation}\label{eq:appendix_state_energy}
    \sum_{t=1}^{n}\|\vx_t\|^2=O(n)
    ,\ \text{a.s.}
\end{equation}
Hence,
\begin{equation}\label{eq:appendix_first_order_upper}
    \lambda_{\max}(n)=O(n)
    ,\ \text{a.s.}
\end{equation}

It remains to prove \eqref{eq:appendix_first_order_lower_bound}. Following a block-excitation construction, let \(q:=N\), \(t_k:=qk\), and
$
A_n
\triangleq
\sum_{k=1}^{n} t_k^{-\epsilon}.
$
Since \(q\) is fixed and \(\epsilon\in[0,1)\), we have
\[
A_n=O(n^{1-\epsilon}),
\qquad
n^{1-\epsilon}=O(A_n).
\]

For each \(k\ge 1\), define
\[
\veta_k
\triangleq
\sum_{j=0}^{q-1}\mG^j\mB\vu_{t_k-1-j},
\qquad
\veta_k^\perp
\triangleq
\vx_{t_k}-\veta_k.
\]
Because the probing inputs on disjoint blocks are independent, \(\{\veta_k\}\) are independent and centered, and each \(\veta_k\) is independent of \(\veta_k^\perp\). Moreover,
\[
\mQ_k
\triangleq
\mathbb{E}\bigl[\veta_k\veta_k^\top\bigr]
=
\sum_{j=0}^{q-1}\frac{1}{(t_k-j)^\epsilon}\mG^j\mB\mB^\top(\mG^j)^\top
\]
\[
\mW_q
\triangleq
\sum_{j=0}^{q-1}\mG^j\mB\mB^\top(\mG^j)^\top.
\]
Since \(q=N\) and \((\mG,\mB)\) is controllable, \(\mW_q\) is positive definite. Therefore there exists \(c_1>0\) such that
 $\lambda_{\min}(\mQ_k)\ge c_1 t_k^{-\epsilon}$
for all sufficiently large \(k\).

Now define
\[
\mH_n
\triangleq
\sum_{k=1}^{n}\veta_k\veta_k^\top,
\qquad
\vU_n
\triangleq
\sum_{k=1}^{n}\veta_k,
\qquad
\mV_n
\triangleq
\sum_{k=1}^{n}\veta_k^\perp\veta_k^\top.
\]
Because \(\{\veta_k\}\) are independent, centered, and satisfy
\(
\mathbb{E}\|\veta_k\|^2=O(t_k^{-\epsilon})
\),
the weighted strong law together with Kronecker's lemma yields
\begin{equation}\label{eq:appendix_Hn_Un}
\mH_n-\sum_{k=1}^{n}\mQ_k=o(A_n),
\
\vU_n=o(A_n),
\
\mV_n=o(A_n)
\
\text{a.s.}
\end{equation}

% To control the cross term, define the block filtration
% \[
% \mathcal{G}_k
% :=
% \sigma\!\left(
%     \veta_1^\perp,\dots,\veta_{k+1}^\perp,\,
%     \veta_1,\dots,\veta_k
% \right).
% \]
% Each entry of \(\mV_n\) is then a martingale transform with weights \(\veta_k^\perp\) and innovations from \(\veta_k\). By \eqref{eq:appendix_state_energy},
% \[
% \sum_{k=1}^{n}\|\veta_k^\perp\|^2=O(n)
% \qquad\text{a.s.}
% \]
% Applying Lemma~\ref{lem:noiseless_md} with \(\alpha=2\) gives
% \begin{equation}\label{eq:appendix_Vn}
% \mV_n=o(A_n)
% \qquad \text{a.s.}
% \end{equation}

Next define the centered block Gramian
\[
\bar{\vx}_n^{(q)}
\triangleq
\frac1n \sum_{k=1}^{n}\vx_{t_k},
\qquad
\mC_n^{(q)}
\triangleq
\sum_{k=1}^{n}
\bigl(
    \vx_{t_k}-\bar{\vx}_n^{(q)}
\bigr)
\bigl(
    \vx_{t_k}-\bar{\vx}_n^{(q)}
\bigr)^\top.
\]
Using \(\vx_{t_k}=\veta_k^\perp+\veta_k\), we obtain
\[
\mC_n^{(q)}
=
\sum_{k=1}^{n}
\bigl(
    \veta_k-\bar{\veta}_n
\bigr)
\bigl(
    \veta_k-\bar{\veta}_n
\bigr)^\top
+
\mR_n,
\]
where \(\bar{\veta}_n=\frac1n\vU_n\) and \(\|\mR_n\|=o(A_n)\) almost surely by \eqref{eq:appendix_Hn_Un} and \eqref{eq:appendix_state_energy}. Since
\[
\sum_{k=1}^{n}
\bigl(
    \veta_k-\bar{\veta}_n
\bigr)
\bigl(
    \veta_k-\bar{\veta}_n
\bigr)^\top
=
\mH_n-n\bar{\veta}_n\bar{\veta}_n^\top
\]
and \(n\bar{\veta}_n\bar{\veta}_n^\top=o(A_n)\), \eqref{eq:appendix_Hn_Un} implies
\[
\lambda_{\min}\!\bigl(\mC_n^{(q)}\bigr)\ge c_2 A_n
\]
almost surely for all sufficiently large \(n\), with some deterministic constant \(c_2>0\).

Now define the full-sample centered Gramian
\[
\bar{\vx}_n
\triangleq
\frac1n\sum_{t=1}^{n}\vx_t,
\qquad
\mC_n
\triangleq
\sum_{t=1}^{n}
\bigl(
    \vx_t-\bar{\vx}_n
\bigr)
\bigl(
    \vx_t-\bar{\vx}_n
\bigr)^\top.
\]

Then 
$
\lambda_{\min}(\mC_n)\ge c_2 n^{1-\epsilon}
$
almost surely for all sufficiently large \(n\). Since
\[
\mS_n^{(z)}
=
\begin{bmatrix}
    1 & 0\\
    \bar{\vx}_n & \mI
\end{bmatrix}
\begin{bmatrix}
    n & 0\\
    0 & \sum_{t=1}^{n}
    \bigl(\vx_t-\bar{\vx}_n\bigr)
    \bigl(\vx_t-\bar{\vx}_n\bigr)^\top
\end{bmatrix}
\begin{bmatrix}
    1 & \bar{\vx}_n^\top\\
    0 & \mI
\end{bmatrix},
\]
and \(\|\bar{\vx}_n\|\) is almost surely bounded by \eqref{eq:appendix_state_energy}, we obtain
\[
\lambda_{\min}\!\bigl(\mS_n^{(z)}\bigr)\ge c_3 n^{1-\epsilon}
\]
almost surely for all sufficiently large \(n\). In particular, for every \(\alpha\in(\frac12,1-\epsilon)\),
\eqref{eq:appendix_first_order_lower_bound} follows. 
\begin{equation}
\frac{\lambda_{\max}(n)}{\lambda_{\min}(n)}
\sqrt{\frac{\log \lambda_{\max}(n)}{\lambda_{\min}(n)}}
\to 0
\
\text{a.s.}
\end{equation}

\textit{Step 2: prove almost-sure consistency.}
Define
\[
\vPhi_k:=\mI\otimes \vz_k,
\qquad
\vtheta:=\operatorname{vec}\!\bigl((\mTheta^\ast)^\top\bigr),
\qquad
\vtheta_{n+1}:=\operatorname{vec}\!\bigl(\widetilde{\mTheta}_{n+1}^\top\bigr).
\]
Then
$
\hat{\vx}_{k+1}
=
\vPhi_k^\top \vtheta+\vw_{k+1}.
$
Since Algorithm~1 uses the RLS recursion with initial condition \(\mP_0=\alpha_0\mI\) and \(\vTheta_0=0\), the corresponding closed-form pilot estimator is
\begin{equation*}
    \mS_n
    \triangleq
    \alpha_0^{-1}\mI+\sum_{k=1}^{n}\vPhi_k\vPhi_k^\top,
    \qquad
    \vtheta_{n+1}
    =
    \mS_n^{-1}
    \sum_{k=1}^{n}\vPhi_k\hat{\vx}_{k+1}.
\end{equation*}
Let
$
\hat{\vtheta}_{n+1}
\triangleq
\operatorname{vec}\!\bigl(\widehat{\vTheta}_{n+1}^\top\bigr).
$
For \(\vbeta\in\mathbb{R}^{N(N+1)}\), define
\begin{equation*}
    \hat{J}_{n+1}(\vbeta)
    \triangleq
    \sum_{k=1}^{n}\|
        \hat{\vx}_{k+1}-\vPhi_k^\top\vbeta
    \|^2
    +
    \lambda_n
    \sum_{l=1}^{N(N+1)}
    \frac{|\vbeta(l)|}{|\hat{\vtheta}_{n+1}(l)|}.
\end{equation*}
Let
\[
\vbeta_{n+1}
\triangleq
\arg\min_{\vbeta\in\mathbb{R}^{N(N+1)}} \hat{J}_{n+1}(\vbeta),
\qquad
\vmu_{n+1}
\triangleq
\vbeta_{n+1}-\vtheta.
\]
Since \(\vbeta_{n+1}=\operatorname{vec}(\vGamma_{n+1}^\top)\), it suffices to show \(\vbeta_{n+1}\to\vtheta\) almost surely.

Let \(A=\{l:\vtheta(l)=0\}\) and \(A^c\) denote its complement. By optimality of \(\vbeta_{n+1}\),
\[
    \hat{J}_{n+1}(\vbeta_{n+1})-\hat{J}_{n+1}(\vtheta)\le 0.
\]
After expansion,
\begin{equation*}
    \begin{aligned}
        0
        \ge\ &
        \sum_{k=1}^{n}
        \Bigl(
            -2\vmu_{n+1}^\top\vPhi_k \vw_{k+1}
            +\|\vPhi_k^\top\vmu_{n+1}\|^2
        \Bigr) \\
        &+
        \lambda_n\sum_{l\in A}
        \frac{|\vbeta_{n+1}(l)|}{|\hat{\vtheta}_{n+1}(l)|}
        +
        \lambda_n\sum_{l\in A^c}
        \frac{|\vbeta_{n+1}(l)|-|\vtheta(l)|}{|\hat{\vtheta}_{n+1}(l)|}.
    \end{aligned}
\end{equation*}
Set
\begin{equation*}
    M_n
    \triangleq
    \sum_{k=1}^{n}
    \Bigl(
        -2\vmu_{n+1}^\top\vPhi_k \vw_{k+1}
        +\|\vPhi_k^\top\vmu_{n+1}\|^2
    \Bigr).
\end{equation*}
Since
\[
\sum_{k=1}^{n}\vPhi_k\vPhi_k^\top
=
\mI\otimes \sum_{k=1}^{n} \vz_k \vz_k^\top,
\]
we have
\begin{equation*}
    \lambda_{\min}(\mS_n)=\alpha_0^{-1}+\lambda_{\min}(n),
    \qquad
    \lambda_{\max}(\mS_n)=\alpha_0^{-1}+\lambda_{\max}(n).
\end{equation*}
Hence \(\lambda_{\min}(\mS_n)\) and \(\lambda_{\max}(\mS_n)\) have the same asymptotic orders as \(\lambda_{\min}(n)\) and \(\lambda_{\max}(n)\), respectively. By a standard self-normalized martingale bound,
\[
    \left\|
    \mS_n^{-1/2}
    \left(\sum_{k=1}^{n}\vPhi_k \vw_{k+1}\right)
    \right\|
    =
    O\!\left(
        \sqrt{\log \lambda_{\max}(\mS_n)}
    \right)
    \quad \text{a.s.}
\]
Hence
\begin{equation*}
    \left|
    -2\vmu_{n+1}^\top\sum_{k=1}^{n}\vPhi_k \vw_{k+1}
    \right|
    \le
    c_1 \lambda_{\max}(n)\|\vmu_{n+1}\|
    \sqrt{\frac{\log \lambda_{\max}(n)}{\lambda_{\min}(n)}}
\end{equation*}
for some deterministic constant \(c_1>0\). Therefore,
\begin{equation*}
    M_n
    \ge
    \lambda_{\min}(n)\|\vmu_{n+1}\|^2
    -
    c_1 \lambda_{\max}(n)\|\vmu_{n+1}\|
    \sqrt{\frac{\log \lambda_{\max}(n)}{\lambda_{\min}(n)}}.
\end{equation*}

Moreover,
$
    \vtheta_{n+1}-\vtheta
    =
    \mS_n^{-1}\sum_{k=1}^{n}\vPhi_k \vw_{k+1}
    -
    \alpha_0^{-1}\mS_n^{-1}\vtheta.
$
The second term is the bias induced by the initial condition \(\mP_0=\alpha_0 \mI\), and it satisfies
$
    \left\|
    \alpha_0^{-1}\mS_n^{-1}\vtheta
    \right\|
    \le
    \frac{\alpha_0^{-1}\|\vtheta\|}{\lambda_{\min}(\mS_n)}
    \to 0
    ,\ \text{a.s.}
$
The first term satisfies
\begin{align*}
\left\|
\mS_n^{-1}\sum_{k=1}^{n}\vPhi_k \vw_{k+1}
\right\|
&\le
\frac{1}{\sqrt{\lambda_{\min}(\mS_n)}}
\left\|
\mS_n^{-1/2}\sum_{k=1}^{n}\vPhi_k \vw_{k+1}
\right\| \\
&=
O\!\left(
\sqrt{\frac{\log \lambda_{\max}(n)}{\lambda_{\min}(n)}}
\right)
\quad\text{a.s.}
\end{align*}
Therefore, \(\vtheta_{n+1}\to\vtheta\) almost surely. In particular, for
\(l\in A^c\), the denominator \(|\hat{\vtheta}_{n+1}(l)|\) is eventually
bounded away from zero. Thus,
\[
    \left|
    \sum_{l\in A^c}
    \frac{|\vbeta_{n+1}(l)|-|\vtheta(l)|}{|\hat{\vtheta}_{n+1}(l)|}
    \right|
    \le c_2 \|\vmu_{n+1}\|
\]
for some deterministic constant \(c_2>0\). Combining the above estimates gives either \(\|\vmu_{n+1}\|=0\) or
\begin{equation*}
    \|\vmu_{n+1}\|
    \le
    c_1\frac{\lambda_{\max}(n)}{\lambda_{\min}(n)}
    \sqrt{\frac{\log \lambda_{\max}(n)}{\lambda_{\min}(n)}}
    +
    c_2\frac{\lambda_n}{\lambda_{\min}(n)}.
\end{equation*}
Define
\begin{equation*}
    \rho_n
    \triangleq
    \frac{\lambda_{\max}(n)}{\lambda_{\min}(n)}
    \sqrt{\frac{\log \lambda_{\max}(n)}{\lambda_{\min}(n)}}
    +
    \frac{\lambda_n}{\lambda_{\min}(n)}.
\end{equation*}
Then \(\|\vmu_{n+1}\|=O(\rho_n)\). Under the choice
\[
\lambda_n^2
=
\lambda_{\max}(n)
\sqrt{\log(\lambda_{\max}(n))\lambda_{\min}(n)},
\]
we have
\[
\frac{\lambda_n}{\lambda_{\min}(n)}
=
\left(
\frac{\lambda_{\max}(n)}{\lambda_{\min}(n)}
\sqrt{\frac{\log \lambda_{\max}(n)}{\lambda_{\min}(n)}}
\right)^{1/2}
\to 0
\quad\text{a.s.}
\]
Hence \(\rho_n\to 0\). Therefore,
$
    \vbeta_{n+1}\to\vtheta
    ,\ \text{a.s.}
$
which is equivalent to \(\widehat{\mTheta}_{n+1}\to\mTheta^\ast\) and in particular, \(\hat{\mG}_{n+1}\to\mG\) almost surely.

\textit{Step 3: prove finite-time exact support recovery.}
Consistency already implies that every truly nonzero entry of \(\vtheta\) is eventually kept nonzero. It remains to exclude false positives on the zero set \(A\).

Suppose, to the contrary, that there exists a subsequence \(\{n_k\}\) such that
\(
\vmu_{n_k+1}(l)\neq 0
\)
for some \(l\in A\). Define
\[
    \overline{\vmu}_{n_k+1}(i)=
    \begin{cases}
        \vmu_{n_k+1}(i), & i\in A^c, \\
        0, & i\in A.
    \end{cases}
\]
Since \(\vtheta+\vmu_{n_k+1}\) minimizes \(\hat{J}_{n_k+1}\),
\[
    \hat{J}_{n_k+1}(\vtheta+\vmu_{n_k+1})
    -
    \hat{J}_{n_k+1}(\vtheta+\overline{\vmu}_{n_k+1})
    \le 0.
\]
Without loss of generality, let \(A^c=\{1,\dots,r\}\) and \(A=\{r+1,\dots, N(N+1)\}\). Partition
the matrix \(\sum_{i=1}^{n_k}\vPhi_i\vPhi_i^\top\) and the vectors
\(\widetilde{\vtheta}_{n_k+1}\triangleq\vtheta-\vtheta_{n_k+1}\) and
\(\vmu_{n_k+1}\) conformably with the decomposition \(A^c\cup A\), and denote
the resulting blocks by \(\mS_{n_k}^{(ij)}\),
\(\widetilde{\vtheta}_{n_k+1}^{(i)}\), and \(\vmu_{n_k+1}^{(i)}\).
The optimality inequality then yields one positive quadratic term on the zero block, together with mixed terms and the weighted penalty on that block. The mixed terms admit the bounds
\begin{equation*}
    \left|
    \vmu_{n_k+1}^{(2)\top}\mS_{n_k}^{(21)}\vmu_{n_k+1}^{(1)}
    \right|
    \le
    c_4\lambda_{\max}(n_k)\|\vmu_{n_k+1}^{(2)}\|\rho_{n_k},
\end{equation*}
\begin{equation*}
    \left|
    2\vmu_{n_k+1}^{(2)\top}\mS_{n_k}^{(22)}\widetilde{\vtheta}_{n_k+1}^{(2)}
    \right|
    \le
    c_5\lambda_{\max}(n_k)\|\vmu_{n_k+1}^{(2)}\|
    \sqrt{\frac{\log \lambda_{\max}(n_k)}{\lambda_{\min}(n_k)}},
\end{equation*}
and
\begin{equation*}
    \left|
    2\vmu_{n_k+1}^{(2)\top}\mS_{n_k}^{(21)}\widetilde{\vtheta}_{n_k+1}^{(1)}
    \right|
    \le
    c_5\lambda_{\max}(n_k)\|\vmu_{n_k+1}^{(2)}\|
    \sqrt{\frac{\log \lambda_{\max}(n_k)}{\lambda_{\min}(n_k)}}.
\end{equation*}
Since \(\vtheta_{n+1}\to\vtheta\), for each \(l\in A\) we have
\[
    c_6
    \sqrt{\frac{\log \lambda_{\max}(n_k)}{\lambda_{\min}(n_k)}}
    \le
    |\hat{\vtheta}_{n_k+1}(l)|
    \le
    c_7
    \sqrt{\frac{\log \lambda_{\max}(n_k)}{\lambda_{\min}(n_k)}}
\]
for some positive constants \(c_6,c_7\). Therefore,
\begin{equation*}
    \lambda_{n_k}
    \sum_{l=r+1}^{N(N+1)}
    \frac{|\vmu_{n_k+1}(l)|}{|\hat{\vtheta}_{n_k+1}(l)|}
    \ge
    \frac{c_8\lambda_{n_k}}
    {\sqrt{\frac{\log \lambda_{\max}(n_k)}{\lambda_{\min}(n_k)}}}
    \|\vmu_{n_k+1}^{(2)}\|
\end{equation*}
for some constant \(c_8>0\).

Combining the above estimates and using
\[
\vmu_{n_k+1}^{(2)\top}\mS_{n_k}^{(22)}\vmu_{n_k+1}^{(2)}
\ge
\lambda_{\min}(n_k)\|\vmu_{n_k+1}^{(2)}\|^2,
\]
we obtain
\begin{equation*}
    \begin{aligned}
        0 \ge\ &
        \lambda_{\min}(n_k)\|\vmu_{n_k+1}^{(2)}\|
        \Bigg[
            \|\vmu_{n_k+1}^{(2)}\| \\
            &-
            c_4\frac{\lambda_{\max}(n_k)}{\lambda_{\min}(n_k)}\rho_{n_k}
            -
            c_5\frac{\lambda_{\max}(n_k)}{\lambda_{\min}(n_k)}
            \sqrt{\frac{\log \lambda_{\max}(n_k)}{\lambda_{\min}(n_k)}} \\
            &+
            \frac{c_8\lambda_{n_k}}
            {\lambda_{\min}(n_k)
            \sqrt{\frac{\log \lambda_{\max}(n_k)}{\lambda_{\min}(n_k)}}}
        \Bigg].
    \end{aligned}
\end{equation*}
By the estimate established at the end of Step 1,
\[
\frac{\lambda_{\max}(n_k)}{\lambda_{\min}(n_k)}
\sqrt{\frac{\log \lambda_{\max}(n_k)}{\lambda_{\min}(n_k)}}
\to 0
\quad\text{a.s.}
\]
Together with the choice of \(\lambda_n\), this implies that, for any sufficiently large \(k\), the positive penalty term dominates the negative terms in the bracket. Hence, the bracket is eventually strictly positive, which contradicts the assumption \(\vmu_{n_k+1}^{(2)}\neq 0\).
%  Therefore,
% \[
% \vmu_{n_k+1}^{(2)}=0
% \]
% for all sufficiently large \(k\) almost surely. In other words, every truly zero entry is removed after a finite random time almost surely. Combined with the consistency already proved on the nonzero entries, this yields finite-time exact support recovery for the block corresponding to \(\mG\). This completes the proof.
\end{proof}

\bibliographystyle{IEEEtran}
\bibliography{myref}

\end{document}